\documentclass[12pt]{amsart}

\usepackage[T1]{fontenc}
\usepackage{lmodern}
\usepackage{amsmath,amssymb,amsthm,mathtools}
\usepackage{enumitem}
\usepackage{microtype}
\usepackage{geometry}
\usepackage[colorlinks, linkcolor=blue, citecolor=red, urlcolor=blue, pagebackref, hypertexnames=false]{hyperref}
\numberwithin{equation}{section}

\newtheorem{theorem}{Theorem}[section]
\newtheorem{proposition}[theorem]{Proposition}
\newtheorem{lemma}[theorem]{Lemma}
\newtheorem{corollary}[theorem]{Corollary}
\theoremstyle{remark}
\newtheorem{remark}[theorem]{Remark}

\newcommand{\R}{\mathbb{R}}

\newcommand{\dd}{\mathrm{d}x}
\newcommand{\dds}{\mathrm{d}s}
\newcommand{\im}{\operatorname{Im}}
\newcommand{\re}{\operatorname{Re}}

\title[Finite-time blow-up for a quadratic NLS system]{Finite-time blow-up for the four-dimensional mass-critical quadratic nonlinear Schr\"odinger system without mass resonance}
	
\author[N.~U.~C.~Nguyen, V.~D.~Dinh]{Ngoc Uyen Cong Nguyen and Van Duong Dinh}

\address[N.~U.~C. Nguyen]{Claire and Leo LLC, Houston, Texas, USA}
\email{contact@uyencong.com}

\address[V.~D.~Dinh]{Claire and Leo LLC, Houston, Texas, USA}
\email{contact@duongdinh.com}
    
\date{August 18, 2026}


\begin{document}

\begin{abstract}
We study the focusing quadratic nonlinear Schr\"odinger system
\[
\begin{cases}
 i\partial_t u+\Delta u=-2v\overline{u},\\
 i\partial_t v+\kappa\Delta v=-u^2,
\end{cases}
\qquad (t,x)\in I\times\mathbb R^4,
\]
where $\kappa>0$. In the non-mass-resonant case $\kappa\neq \frac12$, previous works of Inui--Kishimoto--Nishimura~\cite{IKN} and Dinh--Forcella~\cite{DF} showed that radial solutions with negative energy must either blow up in finite time or exist globally while their $H^1$-norm grows without bound.

In this paper, we prove that every radial $H^1\times H^1$ solution with negative energy blows up in finite time, both forward and backward in time. No finite-variance assumption is required.

The main ingredient is a localized virial argument based on the bounded exponential weight
\[
\nabla\phi_R(x)=2x e^{-|x|^2/R^2}.
\]
A radial weighted interpolation estimate allows us to control the nonlinear error terms by the corresponding weighted kinetic term, up to an $O(R^{-2})$ error depending only on the conserved mass. Moreover, the localized virial quantity itself can be bounded directly in terms of the same weighted kinetic defect. Combining these estimates yields a superlinear Riccati-type differential inequality, which cannot persist for all time and therefore forces finite-time blow-up.
\end{abstract}

\maketitle

\section{Introduction}
\setcounter{equation}{0}
We study the Cauchy problem
\begin{equation}\label{eq:system}
\begin{cases}
 i\partial_t u+\Delta u=-2v\overline{u},\\
 i\partial_t v+\kappa\Delta v=-u^2,\\
 (u,v)|_{t=0}=(u_0,v_0)\in H^1(\R^4)\times H^1(\R^4),
\end{cases}
\end{equation}
where $\kappa>0$.  The value $\kappa=\frac12$ is the mass-resonance condition, while $\kappa\neq\frac12$ is the non-mass-resonant regime.  The system is invariant under the mass-critical scaling
\begin{equation*}
 u_\lambda(t,x)=\lambda^2u(\lambda^2t,\lambda x), \qquad v_\lambda(t,x)=\lambda^2v(\lambda^2t,\lambda x).
\end{equation*}
Its conserved mass and energy are
\begin{align}
 M(u,v)&:=\|u\|_{L^2}^2+2\|v\|_{L^2}^2,\tag{mass}\\
 E(u,v)&:=\frac12\left(\|\nabla u\|_{L^2}^2+\kappa\|\nabla v\|_{L^2}^2\right)
 -\re\int_{\R^4}v\overline{u}^{2}\dd.\tag{energy}
\end{align}
We also write
\begin{equation*}
 T(u,v):=\|\nabla u\|_2^2+\kappa\|\nabla v\|_2^2, \qquad P(u,v):=\re\int_{\R^4}v\overline{u}^{2}\dd,
\end{equation*}
so that $E=\frac12T-P$.

The local $H^1\times H^1$ theory, conservation laws, and blow-up alternative for \eqref{eq:system} are standard; see Hayashi--Ozawa--Tanaka~\cite{HOT} and the references therein.  In particular, for $(u_0,v_0)\in H^1(\R^4) \times H^1(\R^4)$, there is a unique maximal solution
\[
 (u,v)\in C((-T_-,T_+);H^1(\R^4)\times H^1(\R^4)) \cap L^2_{\textrm{loc}}((-T_-,T_+); W^{1,4}(\R^4)\times W^{1,4}(\R^4)),
\]
and if $T_+<\infty$, then
\begin{equation}\label{eq:blowup-alternative}
\limsup_{t\uparrow T_+}\big(\|u(t)\|_{H^1}+\|v(t)\|_{H^1}\big)=\infty,
\end{equation}
with the analogous statement backward in time. Radiality is preserved by uniqueness.

For negative-energy radial data in the non-resonant case, Inui--Kishimoto--Nishimura~\cite{IKN} proved a finite-time blow-up or grow-up alternative, with grow-up along a sequence of times. Dinh--Forcella~\cite{DF} later sharpened this infinite time grow-up by showing that if a radial negative-energy solution were global forward in time, then
\begin{equation}\label{eq:DF-growup}
 T(u(t),v(t))\ge Ct^2
\end{equation}
for all sufficiently large $t$. Their argument therefore leaves open a global branch whose kinetic energy diverges at least quadratically.

The purpose of this paper is to exclude that branch.

\begin{theorem}[Finite-time blow-up]\label{thm:main}
Let $d=4$ and $0<\kappa\neq\frac12$.  Suppose that $(u_0,v_0)\in H^1_{\mathrm{rad}}(\R^4)\times H^1_{\mathrm{rad}}(\R^4)$ satisfy $E(u_0,v_0)<0$. Then both maximal lifespan endpoints are finite:
\[
 T_+<\infty,\qquad T_-<\infty.
\]
In particular,
\[
\limsup_{t\uparrow T_+} (\|u(t)\|_{H^1} + \|v(t)\|_{H^1}) = \infty, \qquad \limsup_{t\downarrow -T_-}(\|u(t)\|_{H^1}+\|v(t)\|_{H^1})=\infty.
\]
\end{theorem}

Let us briefly explain the main idea of the proof. Our argument is based on a localized virial method in the spirit of Ogawa--Tsutsumi \cite{OT}. For a radial weight $\phi_R$, we consider the localized virial quantity
\begin{equation*}
\mathcal M_R(t):= 2\im\int_{\mathbb R^4} \nabla\phi_R(x)\cdot \left(\nabla u(t,x)\overline{u(t,x)}+\nabla v(t,x)\overline{v(t,x)}\right)\dd.
\end{equation*}
It is useful first to recall the mechanism in the previous work of Dinh and Forcella \cite{DF}. There, the multiplier is a truncated quadratic weight: it agrees with $|x|^2$ on $|x|\le R$ and is constant on $|x|\ge 2R$. After estimating the localization errors, the resulting virial inequality has the schematic form
\[
\mathcal M_R'(t) \le 16E(u_0,v_0)-\mathcal K_R(t)+o_R(1),
\]
where $\mathcal K_R(t)\ge0$ is a localized kinetic contribution. For $R$ sufficiently large, the nonpositive term $-\mathcal K_R(t)$ may be discarded, giving
\[
\mathcal M_R'(t)\le -c<0.
\]
Consequently,
\[
\mathcal M_R(t)\lesssim -t
\]
for large $t$. Combining this with the standard estimate
\[
|\mathcal M_R(t)|\lesssim_R T(u(t),v(t))^{1/2}
\]
yields the quadratic grow-up bound \eqref{eq:DF-growup}. Thus the Dinh--Forcella argument shows that any hypothetical global solution must grow up, but it does not exclude such a global branch, since this growth is still compatible with a finite $H^1$ norm at every finite time.

The main new idea is to retain a suitable weighted kinetic defect in the virial inequality and then couple this defect back to the virial quantity itself. To achieve this, we use the exponentially localized quadratic multiplier
\begin{equation*}
\phi_R(x) = R^2\left(1-e^{-|x|^2/R^2}\right), \qquad \nabla\phi_R(x)= 2x e^{-|x|^2/R^2},
\end{equation*}
and measure its deviation from the exact quadratic virial field by
\begin{equation*}
\omega_R(x) := 1-e^{-|x|^2/R^2}.
\end{equation*}
The crucial feature of this choice is that the same defect $\omega_R$ controls three different quantities: the nonlinear localization error, the loss in the principal kinetic term, and the localized virial quantity itself. If 
\[
\mathcal D_R(t):=\int_{\mathbb R^4}\omega_R\left(|\nabla u|^2+\kappa|\nabla v|^2\right)\dd
\] 
denotes the corresponding weighted kinetic defect, the key estimates take the schematic form
\begin{align*}
\int_{\mathbb R^4}\omega_R|v||u|^2\dd &\le \varepsilon \mathcal D_R(t) +O_{M,\varepsilon}(R^{-2}),\\
\mathcal M_R'(t) &\le 16E(u_0,v_0)-c\mathcal D_R(t)+O_M(R^{-2}),\\
|\mathcal M_R(t)| &\lesssim_\kappa R M(u_0,v_0)^{1/2}\mathcal D_R(t)^{1/2}.
\end{align*}
After fixing $R$ sufficiently large, the negative-energy assumption yields
\[
\mathcal M_R'(t) \le -c_0-c_1\mathcal D_R(t)
\]
for some $c_0,c_1>0$. In particular, $\mathcal M_R(t)$ eventually becomes negative. Moreover, we get
\[ 
-\mathcal M_R'(t) \gtrsim  \mathcal D_R(t) \gtrsim  |\mathcal M_R(t)|^2.
\]
Setting
\[
 y(t):=-\mathcal M_R(t)>0,
\]
we therefore obtain the Riccati-type differential inequality
\[
 y'(t)\gtrsim y(t)^2.
\]
Such an inequality cannot persist for all positive times.

\begin{remark}[No finite variance]
Since the field $\nabla \phi_R=2xe^{-|x|^2/R^2}$ is bounded, $\mathcal M_R$ is defined for all $H^1\times H^1$ solutions and no assumption such as $xu_0,xv_0\in L^2$ is required. 
\end{remark}

\begin{remark}[Resonance]
The proof below does not use $\kappa\neq\frac12$; it works for every $\kappa>0$. Thus it also recovers finite-time blow-up for radial negative-energy data in the resonant case $\kappa=\frac12$, where stronger virial cancellations and finite-time blow-up results were already known; see, for example, Dinh~\cite{Dinh, Dinh-2} and the references in~\cite{DF}.
\end{remark}

\begin{remark}[General systems with quadratic interactions]
The argument developed here is not restricted to the two-component system \eqref{eq:system}. It can be adapted, with only minor modifications, to the more general systems of nonlinear Schr\"odinger equations with quadratic interactions considered in \cite{DF,NP}. Indeed, in the four-dimensional mass-critical setting, the nonlinear part of the energy is cubic, and the localized virial identity produces weighted cubic error terms. These can be controlled componentwise by the weighted interpolation estimate of Proposition~\ref{prop:weighted}, together with Young's inequality. At the same time, the geometric estimates for the exponential multiplier are independent of the number of components.

More precisely, if $\mathbf u=(u_1,\ldots,u_l)$ and the system admits the usual positive conserved charge and energy associated with a homogeneous cubic interaction potential, one introduces the weighted kinetic defect
\[
 \mathcal D_R(\mathbf u) := \sum_{j=1}^l a_j \int_{\mathbb R^4} \omega_R(x)|\nabla u_j(x)|^2\dd,
\]
with the positive coefficients $a_j$ determined by the linear part of the system. The same argument then gives, schematically,
\[
 \mathcal M_R'(t) \le c_E E(\mathbf u_0) -c \mathcal D_R(\mathbf u(t)) +O(R^{-2}),
\]
while
\[
 |\mathcal M_R(t)|^2 \lesssim R^2 Q(\mathbf u_0) \mathcal D_R(\mathbf u(t)),
\]
where $Q$ denotes the conserved charge. Consequently, for radial data with negative energy and $R$ sufficiently large, setting $y(t)=-\mathcal M_R(t)$ again yields
\[
 y'(t)\gtrsim y(t)^2.
\]
Thus the global grow-up branch is ruled out by the same Riccati mechanism. In particular, the finite-time blow-up result proved here extends to the corresponding four-dimensional mass-critical general quadratic systems under the standard structural assumptions of \cite{DF,NP}.
\end{remark}

\begin{remark}[Relation to the three-dimensional cubic NLS]
A closely related question arises for the focusing cubic NLS in three dimensions,
\[
 i\partial_t u+\Delta u+|u|^2u=0.
\]
For $H^1$ solutions, without assuming radial symmetry or finite variance, Holmer--Roudenko \cite{HR} proved that either the solution blows up in finite time or there exists a sequence $t_n\to+\infty$ such that
\[
 \|\nabla u(t_n)\|_{L^2}\longrightarrow\infty.
\]
They further formulated the stronger assertion
\[
 \|\nabla u(t)\|_{L^2}\longrightarrow\infty \qquad\text{as }t\to+\infty
\]
as their \emph{weak conjecture}. Thus, in the nonradial setting, the problem is to upgrade grow-up along some diverging sequence to grow-up along every diverging sequence.

For the quadratic system considered here, Dinh--Forcella \cite{DF} obtained, under the additional assumption of radial symmetry, the stronger grow-up estimate \eqref{eq:DF-growup} for all sufficiently large $t$ on any hypothetical global branch. Hence their result gives the analogue of grow-up along every diverging sequence, but only in the radial setting.

The nonradial problem is therefore substantially more delicate, since one must also account for possible spatial translation of the solution. It would be interesting to investigate whether a translation-adapted version of the virial-defect mechanism developed below can shed light on the Holmer--Roudenko weak conjecture for the three-dimensional focusing cubic NLS.
\end{remark}


\section{A radial weighted estimate}\label{sec:weighted}
\setcounter{equation}{0}
Define
\begin{equation*}
\omega_R(x)=1-e^{-(|x|/R)^2}, \qquad R>0,
\end{equation*}
and, for $f\in H^1(\R^4)$,
\begin{equation*}
 D_R(f):=\int_{\R^4}\omega_R(x)|\nabla f(x)|^2\dd.
\end{equation*}

The elementary comparison
\begin{equation}\label{eq:omega-basic}
 c\min\{s^2,1\}\le 1-e^{-s^2}\le \min\{s^2,1\},
 \qquad s\ge0,
\end{equation}
holds with an absolute $c>0$.

\begin{lemma}[One-dimensional local inequality]\label{lem:1d}
Let $I\Subset I^*$ be bounded intervals.  There exists $C=C(I,I^*)$ such that for every $F\in H^1(I^*)$,
\begin{equation}\label{eq:1d}
 \int_I|F|^3\dds
 \le C\left(\int_{I^*}|F'|^2\dds\right)^{1/4}
 \left(\int_{I^*}|F|^2\dds\right)^{5/4}
 +C\left(\int_{I^*}|F|^2\dds\right)^{3/2}.
\end{equation}
\end{lemma}

\begin{proof}
Choose $\eta\in C_c^\infty(I^*)$ with $\eta=1$ on $I$.  The one-dimensional Gagliardo--Nirenberg inequality gives
\[
 \|\eta F\|_{L^\infty} \le C\|\eta F\|_{L^2}^{1/2}\|\partial_s(\eta F)\|_{L^2}^{1/2}.
\]
Therefore
\[
 \int_I|F|^3
 \le \|\eta F\|_\infty\|\eta F\|_2^2
 \le C\|F\|_{L^2(I^*)}^{5/2}
 \big(\|F'\|_{L^2(I^*)}+\|F\|_{L^2(I^*)}\big)^{1/2}.
\]
Using $(a+b)^{1/2}\le a^{1/2}+b^{1/2}$ gives \eqref{eq:1d}.
\end{proof}

\begin{proposition}[Radial quadratic-defect interpolation]\label{prop:weighted}
For every radial $f\in H^1(\R^4)$, every $R>0$, and every $\varepsilon>0$,
\begin{equation}\label{eq:weighted-main}
 \int_{\R^4}\omega_R|f|^3\dd
 \le \varepsilon D_R(f) +C_\varepsilon R^{-2}\left(\|f\|_2^{10/3}+\|f\|_2^3\right).
\end{equation}
The constant is independent of $f$ and $R$.
\end{proposition}

\begin{proof}
We first assume $f\in C_c^\infty(\R^4)$ is radial.  Fix $R>0$ and set
\[
 \rho_j=2^jR, \qquad A_j=\{\rho_j<|x|<2\rho_j\},
 \qquad A_j^*=\{\rho_j/2<|x|<4\rho_j\}, \qquad j\in\mathbb Z.
\]
Let
\begin{equation*}
 d_j:=\min\left\{\left(\frac{\rho_j}{R}\right)^2,1\right\}.
\end{equation*}
By \eqref{eq:omega-basic},
\begin{equation}\label{eq:omega-annulus}
 c d_j\le \omega_R(x)\le C d_j,
 \qquad x\in A_j^*,
\end{equation}
with constants independent of $j$ and $R$.

Write $f=f(r)$ and introduce on the fixed interval $(1/2,4)$
\begin{equation}\label{eq:Uj}
 U_j(s)=\rho_j^2f(\rho_js).
\end{equation}
Set
\begin{align*}
 \mu_j&:=\int_{1/2}^4|U_j(s)|^2s^3\dds,\\
 Q_j&:=\int_{1/2}^4|U_j'(s)|^2s^3\dds.
\end{align*}
Since $s^3\simeq1$ on $(1/2,4)$, Lemma \ref{lem:1d} yields
\begin{equation}\label{eq:fixed-cubic}
 \int_1^2|U_j|^3s^3\dds
 \le C Q_j^{1/4}\mu_j^{5/4}+C\mu_j^{3/2}.
\end{equation}
The scaling in \eqref{eq:Uj} is precisely the $L^2$-critical scaling in four dimensions.  Direct changes of variables give
\begin{align}
 \int_{A_j^*}|f|^2\dd&=|\mathbb S^3|\mu_j,\nonumber\\
 \int_{A_j^*}|\nabla f|^2\dd&=|\mathbb S^3|\rho_j^{-2}Q_j,\label{eq:scale-grad}\\
 \int_{A_j}|f|^3\dd&=|\mathbb S^3|\rho_j^{-2}\int_1^2|U_j|^3s^3\dds.\nonumber
\end{align}
Define the local weighted kinetic defect
\begin{equation*}
 D_j:=\int_{A_j^*}\omega_R|\nabla f|^2\dd.
\end{equation*}
By \eqref{eq:omega-annulus} and \eqref{eq:scale-grad},
\begin{equation}\label{eq:Dj-lower}
 D_j\ge cd_j\rho_j^{-2}Q_j.
\end{equation}
Using \eqref{eq:omega-annulus}, \eqref{eq:fixed-cubic}, and \eqref{eq:Dj-lower}, we obtain
\begin{align}
 \int_{A_j}\omega_R|f|^3\dd
 &\le C d_j\rho_j^{-2} \left(Q_j^{1/4}\mu_j^{5/4}+\mu_j^{3/2}\right)\notag\\
 &\le C d_j^{3/4}\rho_j^{-3/2}D_j^{1/4}\mu_j^{5/4} +C d_j\rho_j^{-2}\mu_j^{3/2}.
 \label{eq:annular-before-young}
\end{align}
Young's inequality with exponents $4$ and $4/3$ implies that for every $\delta>0$,
\begin{equation}\label{eq:young-annulus}
 d_j^{3/4}\rho_j^{-3/2}D_j^{1/4}\mu_j^{5/4} \le \delta D_j+C_\delta d_j\rho_j^{-2}\mu_j^{5/3}.
\end{equation}
The coefficient has the uniform critical-scale bound
\begin{equation}\label{eq:critical-coefficient}
 d_j\rho_j^{-2}\le R^{-2}.
\end{equation}
Indeed, when $\rho_j\le R$ the left side equals $R^{-2}$, while when $\rho_j\ge R$ it equals $\rho_j^{-2}\le R^{-2}$.

Combining \eqref{eq:annular-before-young}--\eqref{eq:critical-coefficient},
\begin{equation}\label{eq:annular-final}
 \int_{A_j}\omega_R|f|^3\dd \le \delta D_j+C_\delta R^{-2} \left(\mu_j^{5/3}+\mu_j^{3/2}\right).
\end{equation}
The enlarged annuli $A_j^*$ have uniformly bounded overlap, so
\begin{equation}\label{eq:overlap}
 \sum_jD_j\le C D_R(f), \qquad \sum_j\mu_j\le C\|f\|_2^2.
\end{equation}
For every $\alpha\ge1$,
\[
 \sum_j\mu_j^\alpha\le\left(\sum_j\mu_j\right)^\alpha.
\]
Summing \eqref{eq:annular-final}, using \eqref{eq:overlap}, and choosing $\delta$ so that the bounded-overlap constant gives the prescribed $\varepsilon$ proves \eqref{eq:weighted-main} for smooth compactly supported radial functions.

Finally, $C_{c,\mathrm{rad}}^\infty(\R^4)$ is dense in $H^1_{\mathrm{rad}}(\R^4)$.  If $f_n\to f$ in $H^1$, then $f_n\to f$ in $L^3$ by interpolation between $L^2$ and the Sobolev embedding $H^1(\R^4)\hookrightarrow L^4(\R^4)$, while $\omega_R^{1/2}\nabla f_n\to\omega_R^{1/2}\nabla f$ in $L^2$.  Passing to the limit proves the proposition for every radial $H^1$ function.
\end{proof}

\begin{corollary}[Mixed cubic defect]\label{cor:mixed}
Let $u,v\in H^1_{\mathrm{rad}}(\R^4)$ and set
\begin{equation*}
 \mathcal D_R(u,v):=\int_{\R^4}\omega_R \big(|\nabla u|^2+\kappa|\nabla v|^2\big)\dd.
\end{equation*}
For every $\varepsilon>0$,
\begin{equation*}
 \int_{\R^4}\omega_R|v||u|^2\dd \le \varepsilon\mathcal D_R(u,v) +C_{\kappa,\varepsilon}R^{-2} \big(m^{5/3}+m^{3/2}\big),
\end{equation*}
where $m=M(u,v)$.
\end{corollary}

\begin{proof}
The pointwise Young inequality
\[
 |v||u|^2\le \frac23|u|^3+\frac13|v|^3
\]
reduces the assertion to Proposition \ref{prop:weighted}.  Apply that proposition to $u$ with a sufficiently small parameter and to $v$ with a parameter smaller by a factor depending on $\kappa$, so that both gradient terms are absorbed into $\varepsilon\mathcal D_R$.  Finally,
\[
 \|u\|_2^2\le m, \qquad \|v\|_2^2\le m/2.
\]
\end{proof}

\section{Geometry of the exponential quadratic multiplier}\label{sec:geometry}

For $R>0$, define
\begin{equation*}
 \phi_R(x)=R^2\left(1-e^{-(|x|/R)^2}\right).
\end{equation*}
Writing $s=|x|/R$ and $h(s)=e^{-s^2}$,
\begin{equation*}
 \phi_R'(r)=2rh(s), \qquad \phi_R''(r)=2h(s)(1-2s^2).
\end{equation*}
For a radial function $\phi(r)$, its Hessian matrix is
\[
D^2\phi = \phi'' e_r \otimes e_r + \frac{\phi'}{r} (I-e_r\otimes e_r), \quad e_r = \frac{x}{r}.
\]
Thus, the radial and tangential eigenvalues of $D^2\phi_R$ are
\begin{equation*}
 \lambda_{\mathrm{rad}}=2h(s)(1-2s^2), \qquad \lambda_{\mathrm{tan}}=2h(s).
\end{equation*}

\begin{lemma}[Defect geometry]\label{lem:geometry}
Let $\omega_R=1-h(|x|/R)$.  Then, as quadratic forms,
\begin{equation}\label{eq:hessian-defect}
 2I-D^2\phi_R\ge 2\omega_R I.
\end{equation}
Moreover, in four dimensions,
\begin{equation}\label{eq:theta-bound}
 0\le 8-\Delta \phi_R \le 12\omega_R.
\end{equation}
Finally,
\begin{equation}\label{eq:linear-derivative}
 \|\Delta^2\phi_R\|_{L^\infty}\le CR^{-2},
\end{equation}
and
\begin{equation}\label{eq:virial-geometry}
 |\nabla\phi_R(x)|^2\le4R^2\omega_R(x).
\end{equation}
\end{lemma}

\begin{proof}
The tangential eigenvalue of $2I-D^2\phi_R$ is $2(1-h)=2\omega_R$, while the radial eigenvalue is
\[
 2-2h(1-2s^2)=2\omega_R+4s^2h\ge2\omega_R,
\]
which proves \eqref{eq:hessian-defect}.  Since $d=4$,
\[
 \Delta\phi_R=\phi_R''+\frac3r\phi_R'=8h-4s^2h.
\]
Thus
\[
8-\Delta \phi_R=8(1-h)+4s^2h=8\omega_R+4s^2h.
\]
The elementary inequality
\begin{equation*}
 s^2e^{-s^2}\le1-e^{-s^2}
\end{equation*}
follows from $e^{s^2}\ge1+s^2$, and gives $0\le 8-\Delta \phi_R \le12\omega_R$, which proves \eqref{eq:theta-bound}.

The scaling $\phi_R(x)=R^2\phi_1(x/R)$ and smoothness of $\phi_1(y)=1-e^{-|y|^2}$ give \eqref{eq:linear-derivative}.  Finally,
\[
 |\nabla\phi_R|^2=4R^2s^2e^{-2s^2} \le4R^2s^2e^{-s^2} \le4R^2\omega_R,
\]
which is \eqref{eq:virial-geometry}.
\end{proof}

\begin{remark}\label{rem:why-exp}
The structural relation \eqref{eq:virial-geometry} is the decisive feature of the exponential multiplier.  A compactly flattened multiplier which equals $|x|^2$ on a ball has zero kinetic defect $2I-D^2\phi_R$ throughout that ball, while $|\nabla\phi_R|$ is nonzero there.  Consequently the localized virial cannot be controlled solely by the same defect that appears with a favorable sign in its derivative.  The exponential profile avoids this mismatch: both quantities vanish at the origin at compatible quadratic order and are linked globally by \eqref{eq:virial-geometry}.
\end{remark}

\section{Virial-defect inequality}\label{sec:coercive}
For a real-valued sufficiently smooth function $\phi$ with bounded derivatives, define
\begin{equation}\label{eq:Mphi}
 \mathcal M_\phi(u,v) :=2\im\int_{\R^4}\nabla\phi(x)\cdot
 \big(\nabla u\overline u+\nabla v\overline v\big)\dd.
\end{equation}
The following identity is standard for \eqref{eq:system}; see Dinh--Forcella~\cite[Section 4]{DF}.

\begin{lemma}[Localized virial identity]\label{lem:virial-id}
Let $(u,v)$ be a $H^1\times H^1$ solution of \eqref{eq:system} defined on the maximal time interval $(-T_-,T_+)$. Then for all $t\in (-T_-, T_+)$,
\begin{align}
 \frac{d}{dt}\mathcal M_\phi(u,v)
 ={}&-\int_{\R^4}\Delta^2\phi\big(|u|^2+\kappa|v|^2\big)\dd\label{eq:virial-id}\\
 &+4\re\sum_{j,k=1}^4\int_{\R^4}\phi_{jk}
 \big(\partial_ju \partial_k\overline u + \kappa \partial_jv \partial_k\overline v\big)\dd\notag\\
 &-2\re\int_{\R^4}\Delta \phi v\overline u^{2}\dd.\notag
\end{align}
In particular, for $\phi(x)=|x|^2$ one has
\[
 \frac{d}{dt}\mathcal M_{|x|^2}(u,v)=16E(u,v).
\]
\end{lemma}

Set
\begin{equation*}
 \mathcal M_R(t):=\mathcal M_{\phi_R}(u(t),v(t)).
\end{equation*}
The boundedness of $\nabla\phi_R$ implies that $\mathcal M_R(t)$ is finite for every $H^1\times H^1$ state.

\begin{proposition}[Virial-defect inequality]\label{prop:virial}
For every $(u,v)\in H^1(\R^4)\times H^1(\R^4)$,
\begin{equation}\label{eq:virial}
 |\mathcal M_R(u,v)| \le C_\kappa RM(u,v)^{1/2}\mathcal D_R(u,v)^{1/2}.
\end{equation}
\end{proposition}

\begin{proof}
Using \eqref{eq:virial-geometry} and Cauchy--Schwarz,
\begin{align*}
 |\mathcal M_R(u,v)|
 &\le2\int|\nabla\phi_R|
 \big(|u||\nabla u|+|v||\nabla v|\big)\dd\\
 &\le4R\|u\|_2\left(\int\omega_R|\nabla u|^2\right)^{1/2}
 +4R\|v\|_2\left(\int\omega_R|\nabla v|^2\right)^{1/2}.
\end{align*}
Since
\[
 \|u\|_2^2\le M(u,v), \qquad \|v\|_2^2\le\frac12M(u,v),
\]
and
\[
 \int\omega_R|\nabla u|^2\le\mathcal D_R(u,v),
 \qquad
 \int\omega_R|\nabla v|^2\le\kappa^{-1}\mathcal D_R(u,v),
\]
we obtain \eqref{eq:virial}.
\end{proof}

\begin{proposition}[Coercive localized virial inequality]\label{prop:coercive}
Let $(u,v)$ be a radial $H^1\times H^1$ solution of \eqref{eq:system} defined on the maximal time interval $(-T_-,T_+)$. There exist constants $c_\kappa,C_\kappa>0$ such that for every $R>0$ and every $t\in (-T_-,T_+)$,
\begin{equation}\label{eq:coercive}
 \mathcal M_R'(t)
 \le16E(u_0,v_0)-c_\kappa\mathcal D_R(u(t),v(t)) +C_\kappa R^{-2}G(m),
\end{equation}
where
\begin{equation*}
 G(m):=m+m^{3/2}+m^{5/3},
\end{equation*}
with $m:=M(u_0,v_0)=M(u(t),v(t))$.
\end{proposition}

\begin{proof}
Starting from \eqref{eq:virial-id}, add and subtract the exact quadratic-virial terms.  Since $d=4$ and $16E=8T-16P$, we obtain the exact decomposition
\begin{align}
 \mathcal M_R'(t)
 ={}&16E(u_0,v_0)\label{eq:exact-decomp}\\
 &-4\re\sum_{j,k=1}^4\int_{\R^4}
 (2\delta_{jk}-(\phi_R)_{jk})
 \big(\partial_ju\partial_k\overline u+\kappa\partial_jv\partial_k\overline v\big)\dd\notag\\
 &+2\re\int_{\R^4}(8-\Delta\phi_R)v\overline u^{2}\dd\notag\\
 &-\int_{\R^4}\Delta^2\phi_R\big(|u|^2+\kappa|v|^2\big)\dd.\notag
\end{align}
By \eqref{eq:hessian-defect}, the principal term satisfies
\begin{equation}\label{eq:principal}
 -4\re\sum_{j,k=1}^4\int_{\R^4}
 (2\delta_{jk}-(\phi_R)_{jk})
 \big(\partial_ju\partial_k\overline u+\kappa\partial_jv\partial_k\overline v\big)\dd
 \le -8\mathcal D_R(u,v).
\end{equation}
For the nonlinear defect, \eqref{eq:theta-bound} and Corollary \ref{cor:mixed} give, for any $\eta>0$,
\begin{align}
 \left|2\re\int(8-\Delta\phi_R)v\overline u^{2}\dd\right|
 &\le24\int\omega_R|v||u|^2\dd\notag\\
 &\le \eta\mathcal D_R(u,v)
 +C_{\kappa,\eta}R^{-2}\big(m^{5/3}+m^{3/2}\big).
 \label{eq:nl-defect}
\end{align}
By \eqref{eq:linear-derivative} and mass conservation,
\begin{equation}\label{eq:linear-defect}
 \left|\int\Delta^2\phi_R(|u|^2+\kappa|v|^2)\dd\right|
 \le C_\kappa R^{-2}m.
\end{equation}
Choose, for example, $\eta=4$ in \eqref{eq:nl-defect}.  Combining \eqref{eq:exact-decomp}--\eqref{eq:linear-defect} gives \eqref{eq:coercive}, with $c_\kappa$ replaceable by an absolute positive constant if $\kappa$ is regarded as fixed in the definition of $\mathcal D_R$.
\end{proof}

\section{Proof of the main theorem}\label{sec:riccati}

We now complete the argument.

\begin{proof}[Proof of Theorem \ref{thm:main}]
Write
\begin{equation*}
 E(u_0,v_0)=-e, \qquad e>0.
\end{equation*}
Fix $R>0$ sufficiently large that
\begin{equation*}
 C_\kappa R^{-2}G(m)\le8e,
\end{equation*}
where $C_\kappa$ is the constant in Proposition \ref{prop:coercive}.  Then \eqref{eq:coercive} yields
\begin{equation}\label{eq:MR-negative}
 \mathcal M_R'(t)\le-8e-c_\kappa\mathcal D_R(t),
\end{equation}
where $\mathcal D_R(t)=\mathcal D_R(u(t),v(t))$.

Assume for contradiction that $T_+=\infty$.  Integrating the constant negative term in \eqref{eq:MR-negative},
\begin{equation*}
 \mathcal M_R(t)\le\mathcal M_R(0)-8et.
\end{equation*}
Hence there exists $t_0>0$ such that
\[
 \mathcal M_R(t)<0, \qquad t\ge t_0.
\]
Define
\begin{equation*}
 y(t):=-\mathcal M_R(t)>0,
 \qquad t\ge t_0.
\end{equation*}
By \eqref{eq:MR-negative},
\begin{equation}\label{eq:yprime-defect}
 y'(t)\ge c_\kappa\mathcal D_R(t).
\end{equation}
On the other hand, Proposition \ref{prop:virial} and mass conservation give
\[
 y(t)\le C_\kappa Rm^{1/2}\mathcal D_R(t)^{1/2}.
\]
Since $m>0$, this implies
\begin{equation}\label{eq:defect-lower-y}
 \mathcal D_R(t)\ge c_\kappa R^{-2}m^{-1}y(t)^2.
\end{equation}
Combining \eqref{eq:yprime-defect} and \eqref{eq:defect-lower-y},
\begin{equation*}
 y'(t)\ge K y(t)^2, \qquad K=c_\kappa R^{-2}m^{-1}>0,
 \qquad t\ge t_0.
\end{equation*}
Therefore
\begin{equation*}
 \frac{d}{dt}\frac1{y(t)}=-\frac{y'(t)}{y(t)^2}\le-K.
\end{equation*}
Integrating,
\begin{equation}\label{eq:inverse-y-integrated}
 \frac1{y(t)}\le\frac1{y(t_0)}-K(t-t_0).
\end{equation}
The right-hand side vanishes at the finite time
\begin{equation*}
 T^\sharp:=t_0+\frac1{K y(t_0)}.
\end{equation*}
Thus \eqref{eq:inverse-y-integrated} forces $y(t)\to\infty$ as $t\uparrow T^\sharp$.

This contradicts the global-existence hypothesis.  Indeed, a global $H^1\times H^1$ solution is continuous on the compact interval $[0,T^\sharp]$, hence bounded there in $H^1\times H^1$; since $\nabla\phi_R\in L^\infty$, the definition \eqref{eq:Mphi} then gives $\sup_{0\le t\le T^\sharp}|\mathcal M_R(t)|<\infty$.  Therefore $T_+<\infty$, and the forward norm divergence follows from the blow-up alternative \eqref{eq:blowup-alternative}.

For the backward direction, define
\[
 \widetilde u(t,x)=\overline{u(-t,x)}, \qquad \widetilde v(t,x)=\overline{v(-t,x)}.
\]
Then $(\widetilde u,\widetilde v)$ solves \eqref{eq:system}, is radial, and has the same mass and energy as $(u,v)$.  Applying the forward argument to $(\widetilde u,\widetilde v)$ gives $T_-<\infty$.  This completes the proof.
\end{proof}

\end{document}